\documentclass[a4paper,12pt,draft]{amsart}
\usepackage[margin=30mm]{geometry}
\usepackage{amssymb, amsmath, amsthm, amscd}

\usepackage[T1]{fontenc}
\usepackage{eucal,mathrsfs,dsfont}
\usepackage{color}

\renewcommand{\le}{\leqslant}
\renewcommand{\ge}{\geqslant}

\definecolor{mno}{rgb}{0.5,0.1,0.5}

\newcommand{\R}{\mathds R}

\newcommand{\Pp}{\mathds P}
\newcommand{\Ee}{\mathds E}

\newcommand{\I}{\mathds 1}

\newcommand{\Ss}{ \mathds{S}}
\newcommand{\D}{\mathscr{D}}
\newcommand{\E}{\mathscr{E}}

\newtheorem{theorem}{Theorem}[section]
\newtheorem{lemma}[theorem]{Lemma}
\newtheorem{proposition}[theorem]{Proposition}
\newtheorem{corollary}[theorem]{Corollary}

\theoremstyle{definition}

\newtheorem{example}[theorem]{Example}
\newtheorem{remark}[theorem]{Remark}

\begin{document}
\allowdisplaybreaks
\title[Intrinsic Ultracontractivity for General L\'evy processes]
{\bfseries Intrinsic Ultracontractivity for General L\'evy processes on
Bounded Open Sets}
\author{Xin Chen\qquad Jian Wang}
\thanks{\emph{X.\ Chen:}
   Department of Mathematics, Shanghai Jiao Tong University, 200240 Shanghai, P.R. China. \texttt{chenxin\_217@hotmail.com}}
  \thanks{\emph{J.\ Wang:}
   School of Mathematics and Computer Science, Fujian Normal University, 350007 Fuzhou, P.R. China. \texttt{jianwang@fjnu.edu.cn}}

\date{}

\maketitle

\begin{abstract} We prove that a general (not necessarily symmetric) L\'{e}vy process
 killed on exiting a bounded open set (without regular condition on the boundary) is intrinsically ultracontractive, provided that $B(0,R_0)\subseteq \rm{supp}(\nu)$
for some constant $R_0>0$, where $\rm{supp}(\nu)$ denotes the support of the associated L\'{e}vy measure $\nu$. For a symmetric L\'{e}vy process killed on exiting a bounded H\"{o}lder domain of order $0$, we also obtain the intrinsic ultracontractivity under much weaker assumption on the associated L\'{e}vy measure.
\medskip

\noindent\textbf{Keywords:} L\'{e}vy process; Dirichlet semigroup;
intrinsic ultracontractivity; super Poincar\'{e} inequality
\medskip

\noindent \textbf{MSC 2010:} 60G51; 60G52; 60J25; 60J75.
\end{abstract}
\allowdisplaybreaks

\section{Introduction and Main Results}\label{section1}
\subsection{Dirichlet Semigroup and its Dual Semigroup for General L\'{e}vy Process}\label{section11}
Let $X=((X_t)_{t\ge0}, \Pp^x)$ be a L\'evy process on $\R^d$ with L\'evy triplet
$(Q,b,\nu)$, such that its characteristic exponent is given by
\begin{equation}\label{sym} q(\xi)=
\frac{1}{2}\langle \xi, Q\xi\rangle+i \langle \xi, b \rangle+\int_{\R^d \setminus \{0\}}
\big(1-e^{i\langle\xi, z\rangle}+i\langle\xi, z\rangle\I_{\{|z|\le 1\}}\big)\,\nu(dz),\quad \xi\in\R^d,
\end{equation} where $Q:\R^d \rightarrow \R^d$ is a symmetric non-negative definite $d\times d$ matrix, $b\in\R^d$, and $\nu$ is a L\'{e}vy measure on $\R^d$.
Let
$\hat X=(\hat X_t)_{t \ge 0}$ denote the dual process of $X$, which is a L\'evy process with the L\'evy triplet
$(Q,-b,\hat\nu)$ such that $\hat \nu(U)=\nu(-U)$ for any $U \in \mathscr{B}(\R^d).$
Throughout this paper, we assume that the process $X$ has a continuous, bounded and strictly positive transition density $p(t,x,y)=p(t,0,y-x)$, i.e.
for every $t>0$ and $f \in {B}_b(\R^d)$, (here and in what follows, ${B}_b(\R^d)$ denotes the set of bounded measurable functions on $\R^d$,)
\begin{equation*}
\Ee^x f(X_t)=\int_{\R^d} p(t,x,y)f(y)\,dy,\quad x \in \R^d,
\end{equation*} $p(t,\cdot,\cdot): \R^d \times \R^d \mapsto (0,\infty)$ is continuous, and there is a constant $c(t)>0$ such that
$$0<p(t,x,y)\le c(t),\ \ \forall\ x,y \in \R^d.$$ See \cite{BSW, KSz, KSc1, KSc2, SSW, S} for sufficient conditions in terms of characteristic exponent $q(\xi)$.

Let $$T_tf(x)=\Ee^xf(X_t),\quad \hat{T}_tf(x)=\Ee^xf(\hat X_t).$$ Then for any non-negative Borel measurable function $f$ and $g$,
$$\int T_tf(x)g(x)\,dx = \int f(x)\hat T_tg(x)\,dx.$$ Hence, the (dual) L\'{e}vy process $\hat X$ also possesses a continuous, bounded and strictly positive transition density $\hat p(t,x,y)$ such that for any $t>0$ and $x,y\in\R^d$, $\hat p(t,x,y)=p(t,y,x)$ and
\begin{equation*}\label{e1-0a}
\Ee^x f(\hat X_t)=\int_{\R^d}\hat p(t,x,y)f(y)\,dy=\int_{\R^d} p(t,y,x)f(y)\,dy, \quad x \in \R^d, f\in \mathscr{B}_b(\R^d).
\end{equation*}

\ \

Let $D\subseteq \R^d$ be an open set. Define the following subprocess of $X$
\begin{equation}\label{e1-1a}
X^D_t:=\begin{cases}
X_t,\quad\text{if}\ t<\tau_D\\
\,\partial,\,\,\quad \text{if}\ t\ge \tau_D,
\end{cases}
\end{equation}
where $\tau_D:=\inf \{t > 0: X_t \notin D\}$ and $\partial$ denotes the cemetery point. Then, the process $X^D:=(X^D_t)_{t\ge0}$
is called the killed process of $X$ on exiting $D$. By the strong Markov property and the continuity of $p(t,\cdot,\cdot)$ for all $t>0$, it is easy to see that the process $X^D$ has a transition density (or Dirichlet heat kernel) $p^D(t,x,y)$, which enjoys the following relation with $p(t,x,y)$:
\begin{equation}\label{e1-0}
\begin{split} p^D(t,x,y)&=p(t,x,y)-\Ee^x\big[p(t-\tau_D,X_{\tau_D},y)\I_{\{t\ge \tau_D\}}\big],\quad x, y \in D;\\
p^D(t,x,y)&=0,\quad x \notin D \text{ or } y \notin D.
\end{split}
\end{equation}
According to \eqref{e1-0}, one can show that $p^D(t,x,y)$, $t>0$, satisfy the Chapman-Kolmogorov equation; moreover, for
every $t>0$ the function $p^D(t,\cdot,\cdot): D \times D \mapsto [0,\infty)$ is continuous, and $\sup_{x,y \in D}p^D(t,x,y)\le \sup_{x,y\in\R^d}p(t,x,y)<\infty$, see e.g.\ the proof of \cite[Theorem 2.4]{CZ}. Define
\begin{equation*}
T_t^D f(x)=\Ee^x f(X_t^D)=\int_D p^D(t,x,y)f(y)\, dy,\quad t>0, x\in D, f\in L^2(D;dx).
\end{equation*}
It is a standard result that $(T_t^D)_{t \ge 0}$ is a strongly continuous contraction semigroup on
$L^2(D;dx)$, which is called the Dirichlet semigroup associated with the process
$X^D$. We further assume that $p^D(t,x,y)>0$ for every $t>0$ and $x,y \in D$, which is equivalent to saying that $(T_t^D)_{t \ge 0}$ is irreducible, i.e. $T_t^D (\I_U)(x)> 0$ for every $t>0$, $x\in D$ and
open set $U\subseteq D$ with $|U|>0$, where $|U|$ denotes the Lebesgue measure of $U$. We should mention that even if the transition density $p(t,x,y)$ is smooth and strictly positive, it is non-trivial to show the strict positivity of $p^D(t,x,y)$, see Proposition \ref{p2-1} below for some mild assumption on the L\'{e}vy measure.

Let $\hat \tau_D:=\inf \{t > 0: \hat X_t \notin D\}$ be the first exit time
from $D$ for the dual process $\hat X$. Similar to \eqref{e1-1a}, we can define the killed process $\hat X^D:=(\hat X^D_t)_{t\ge0}$ of $\hat X$  on exiting $D$. For any $t>0$ and $x\in D$, define $$\hat{T}_t^D f(x)=\Ee^x f(\hat{X}_t^D).$$ Due to Hunt's switching identity (see \cite[Chapter II, Theorem 5]{Bo}),
$$\int_D f(x) T_t^D g(x)\,dx=\int_D g(x) \hat T_t^D f(x)\,dx.$$
Then, the killed process $\hat X^D$ also has a transition density $\hat p^D(t,x,y)$ such that
$\hat p^D(t,x,y)=p^D(t,y,x)$ for all $t>0$ and $x,y\in D$, and so
\begin{equation*}
\begin{split}
\hat T_t^D f(x)=\int_D  p^D(t,y,x)f(y)\, dy,\quad t>0, x\in D, f\in L^2(D;dx).
\end{split}
\end{equation*}

\ \

When the L\'{e}vy process $X$ is symmetric, the associated L\'{e}vy measure $\nu$ is symmetric, and the characteristic exponent $q(\xi)$ given by \eqref{sym} is reduced into
$$q(\xi)= \frac{1}{2}\langle \xi, Q\xi\rangle+\int\big( 1-\cos \langle \xi,z\rangle\big)\,\nu(dz).$$ Then, $(T_t)_{t\ge 0}$ and $(T_t^D)_{t\ge0}$ are symmetric semigroups on $L^2(\R^d;dx)$ and $L^2(D;dx)$, respectively. In particular, $T_t=\hat{T}_t$ and $T_t^D=\hat{T}_t^D$ for any $t>0$, and $p^D(t,x,y)=\hat{p}^D(t,x,y)$ for any $t>0$ and $x,y\in D.$

\subsection{Main Result}
In this part, we always assume that \emph{$D$ is a bounded open subset of $\R^d$}.  Since $\sup_{x,y \in D}p^D(t,x,y)<\infty$ and
$D$ is bounded, both $T_t^D$ and $\hat T_t^D$ are Hilbert-Schmidt operators on $L^2(D;dx)$ for every $t>0$, and so they are
compact. Noticing that $p^D(t,x,y)>0$ for all $x,y\in D$, it follows form Jentzsch's Theorem (see \cite[Chapter V, Theorem 6.6]{Sc}) that the common value $-\lambda_1=\sup \textrm{Re}(\sigma(L_D))=\sup \textrm{Re}(\sigma (\hat L_D))<0$
\footnote{{For any $f,g\in C_c^\infty(D;dx)$, $$\int_D f(x)L_Dg(x) \,dx=\int_D g(x) \hat{L}_D f(x)\,dx.$$ Since the associated L\'{e}vy measure $\hat \nu $ of ${\hat L}_D$ satisfies that $\hat \nu(U)=\nu(-U)$ for any $U \in \mathscr{B}(\R^d),$ we have for $f \in C_c^{\infty}(D;dx)$
\begin{align*}-\int_D f(x) L_D f(x) \,dx=&-\frac{1}{2}\int_D f(x) (L_D+\hat L_D)f(x) \,dx\\
 =& \frac{1}{4}\int_{\R^d} \int_{\R^d} \Big(f(x+z)-f(x)\Big)^2\,\big(\nu(dz)+\hat \nu(dz)\big)\,dx.\end{align*}
Now, let $\phi_1$ be the normalized non-zero eigenfunction associated with
 $\lambda_1$. Since $L_D\phi_1=-\lambda_1 \phi_1$ and $\phi_1\equiv0$ on $D^c$, by the standard approximation, \begin{align*}\lambda_1=&-\int_D \phi_1(x) L_D\phi_1(x)\,dx
 = \frac{1}{4}\int_{\R^d} \int_{\R^d} \Big(\phi_1(x+z)-\phi_1(x)\Big)^2\,\big(\nu(dz)+\hat \nu(dz)\big)\,dx>0;\end{align*}
  otherwise $\phi_1$ is a constant function on $\R^d$, which is impossible.}} is an eigenvalue of multiplicity 1 for the operators $L_D$ and $\hat L_D$, which are $L^2(D;dx)$-generators of $(T^D_t)_{t\ge0}$ and $(\hat T^D_t)_{t\ge0}$ respectively. Moreover, according to \cite[Proposition 3.8]{KS1}, the corresponding eigenfunctions $\phi_1$ and $\hat\phi_1$ can be chosen to be bounded, continuous and strictly positive on $D$. In the literature, this eigenfunction
$\phi_1$ (resp. $\hat \phi_1$) is named ground state (resp.\ dual ground state). We are interested in the intrinsic
ultracontractivity of $(T_t^D)_{t\ge0}$, which is defined that for every $t>0$, there exists a constant $C(t)>0$ such that
\begin{equation}\label{e1-3}
p^D(t,x,y)\le C(t)\phi_1(x)\hat \phi_1(y),\quad \ x,y\in D.
\end{equation}

The notion of intrinsic ultracontractivity for symmetric semigroups was first introduced by Davies and Simon in
\cite{DS} (note that in symmetric setting, $\phi_1=\hat{\phi}_1$ in \eqref{e1-3}), and then it was generalized to non-symmetric semigroups by Kim and Song in \cite{KS2}.
It has wide applications in the area of analysis and probability. Recently,
the intrinsic ultracontractivity of Markov semigroups (including Dirichlet semigroups and Feyman-Kac semigroups) has been intensively established for various L\'evy processes or L\'evy type processes, see e.g. \cite{CW1, CW2,CS,CS1,G,KK,KL,KS2, KS3, KS1,K,KS,Kw}.
The aim of this paper is to study the intrinsic ultracontractivity of Dirichlet semigroup $(T_t^D)_{t \ge 0}$ for a discontinuous
(not necessarily symmetric) L\'{e}vy process (which may contain Brownian motion) on a bounded open set $D$ with very mild conditions on its L\'evy measures $\nu$ and the set $D$.

\ \

To state our first contribution, we need the following additional assumption on the L\'{e}vy measure $\nu$.
\begin{itemize}
\item[{\bf (A1)}] \emph{There exists a constant $R_0>0$ such that
\begin{equation}\label{e1-1}
B(0,R_0)\subseteq \rm {supp}(\nu)
\end{equation}
where $B(x,r)$ denotes the ball
with center $x \in \R^d$ and radius $r>0$, and $\rm {supp}(\nu)$ denotes the support of the L\'{e}vy measure $\nu$.}
\end{itemize}

Note that for  L\'{e}vy process with finite range jumps, the distance between connected components of $D$ should not be too far away, otherwise $p^D(t,x,y)$ will be zero there. Therefore, to ensure the strictly positivity of $p^D(t,x,y)$, we need the following \emph{roughly connected} assumption on the open set $D$, e.g.\ see \cite[Definition 4.3]{KS1}.
\begin{itemize}
\item[{\bf (RC)}] \emph{For any $x,y\in D$, there exist distinct connected components $\{D_i\}_{i=1}^m$ of $D$, such that
$x\in D_1$, $y\in D_m$ and for every
$1 \le i \le m-1$, $dist(D_i, D_{i+1})< R_0$, where $R_0$ is the constant in
Assumption {\bf (A1)}.}
\end{itemize}

\begin{theorem}\label{t1-1}
Let $X$ be the L\'{e}vy process as above such that assumption {\bf({A1})} holds, and suppose that the open set $D$ satisfies {\bf({RC})}.
Then the associated Dirichlet semigroup $(T_t^D)_{t \ge 0}$ is
intrinsically ultracontractive. More explicitly, there is a constant $c>0$ such that for all $t>0$ and $x,y\in D,$
\begin{equation}\label{t1-1-1-1} p^D(t,x,y)\le  \frac{ce^{-\lambda_1t} }{(t\wedge1)^2}\bigg(\int e^{-(t\wedge1)\mathrm{Re}\, q(\xi)}\,d\xi\bigg) \phi_1(x)\hat\phi_1(y),\end{equation} where $q(\xi)$ is the characteristic exponent of the process $X$ given by \eqref{sym}, and $-\lambda_1<0$ is the common eigenvalue corresponding to ground state $\phi_1$ and $\hat \phi_1$.
\end{theorem}

For symmetric L\'{e}vy process,
\cite{G} has established
the intrinsic ultracontractivity of Dirichlet semigroup on any bounded open set $D$,
when the associated L\'{e}vy measure has full support (i.e.\ \eqref{e1-1} holds with $R_0=\infty$). For general L\'{e}vy process, if the Lebesgue measure
is absolutely continuous with respect to L\'evy measure, the intrinsic ultracontractivity of Dirichlet semigroup on any bounded open set $D$
was verified in \cite{KS1}.
Note that in the latter setting the associated L\'evy measure also has full support, and the corresponding L\'evy process has full range jumps.
The reader can refer to \cite{G} for other non-degenerate conditions on L\'evy measure in the symmetric setting.
 On the other hand, when L\'evy measure is compactly supported and the Radon-Nikodym derivative of absolutely continuous part of L\'evy measure is bounded below by some positive constant near the origin,
Kim and Song proved in \cite{KS1} that the corresponding Dirichlet semigroup is intrinsically ultracontractive for general (not necessarily symmetric) L\'{e}vy process provided that $D$ is $\kappa$-fat, see \cite[Assumption A4(b)]{KS1}.  The reader can also refer to \cite{C} for the intrinsic ultracontractivity
for the Dirichlet semigroup associated with a Brownian motion on different non-smooth domains.

The new point of Theorem \ref{t1-1} is due to that, it gets rid of any regularity condition on bounded open set $D$ to ensure the intrinsic ultracontractivity of associated Dirichlet semigroups for general L\'{e}vy process with finite range jumps. Besides, we do not require that L\'evy measure has an absolutely continuous part. See Example \ref{ex1-1-2} in the end of Section \ref{section3} for an application of Theorem \ref{t1-1}.

\subsection{Symmetric L\'{e}vy Process on Bounded H\"older Domain of Order $0$}\label{section13}
Throughout the paper, we always refer to a connected open set as a domain.
It is known that the intrinsic ultracontractivity of Dirichlet semigroups for Brownian motion on a bounded domain $D$ depends on the
geometry of the boundary of $D$ (see \cite{C}). Theorem \ref{t1-1} indicates that for L\'{e}vy process even with finite range jumps, the associated Dirichlet semigroup can be intrinsically  ultracontractive without any regularity condition on the bounded domain. In fact, for more general bounded domains including bounded H\"older domain of order $0$, we can prove the intrinsic ultracontractivity of the associated Dirichlet semigroups for symmetric  L\'{e}vy process, whose L\'{e}vy measure satisfies weaker assumption than {\bf (A1)}.

To be more explicit, we introduce the following \emph{logarithmic distance integrability} assumption on the domain $D$.
\begin{itemize}
\item[{\bf (LDI)}] \emph{{ For each $\theta>0$,
\begin{equation}\label{e1-4}
\int_D \Big|\log\Big(\frac{1}{\rho_{\partial D}(x)}\Big)\Big|^{\theta}\,dx <\infty,
\end{equation}
where $\rho_{\partial D}(x)=\inf\{|x-y|: y \in \partial D\}$
denotes the distance between $x$ and the boundary of $D$.}}
\end{itemize}

According to the proof of \cite[Theorem 2]{SS}, any H\"older domain of order $0$ satisfies
{\bf (LDI)}. Note that, it is shown in \cite{SS} that a H\"{o}lder domain of order $0$ is bounded. John domains, in particular bounded Lipschitz
domains, are H\"{o}lder domains of order $0$.
In fact, recall that a domain $D$ is called \emph{H\"older domain of order $0$} if there exist
some constants $c_1,c_2>0$ and $x_0 \in D$, such that
\begin{equation*}
k_D(x_0,x)\le c_1\log\Big(\frac{1}{\rho_{\partial D}(x)}\Big)+c_2,\quad \forall\ x \in D.
\end{equation*}
Here,
$k_D(x,y)$ is the hyperbolic distance between
$x, y \in D$ defined by
\begin{equation*}
k_D(x,y):=\inf_{\gamma}\int_{0}^1\frac{|\dot{\gamma}(s)|}{\rho_{\partial D}(\gamma(s))}\,ds,
\end{equation*}
where the infimum is taken over all the rectifiable curves $\gamma: [0,1]$
$\rightarrow D$ such that $\gamma(0)=x$ and $\gamma(1)=y$. On the one hand, $$k_D(x_0,x)\le c_3 m$$ on
$$D_m:=\bigcup \Big\{Q\in {\mathscr{W}}:\ \frac{b^{-1}}{2^m}\le {\rm {diam}} (Q)\le \frac{b}{2^m}\Big\}$$
for every $m\ge 1$ and some constants $c_3, b>1$, where
${\mathscr{W}}=\{Q\}$ is a Whitney decomposition of $D$ into closed dyadic cubes with disjoint interiors, and
${\rm {diam}} (Q)$ denotes the diameter for a cube $Q\in {\mathscr{W}}$.
 Then, following the argument in \cite[Page 76]{SS} of \cite[Theorem 2]{SS}, one can see that for each $\theta>0$,
$$\int_D k^{\theta}_D(x_0,x)\,dx<\infty.$$
On the other hand, according to \cite[Line 17 in Page 76]{SS},
there is a constant $c_4>0$ such that for every $m \ge 1$
$$\log\Big(\frac{1}{\rho_{\partial D}(x)}\Big)\le c_4m,\quad x \in D_m.$$
From these, we
can repeat the proof of \cite[Theorem 2]{SS} and obtain that \eqref{e1-4} holds any H\"older domain of order $0$.

In the remainder of this subsection, we further assume that the L\'{e}vy process $X$ is symmetric, and adopt the following assumption on the L\'{e}vy measure $\nu$:
\begin{itemize}
\item[{\bf (A2)}] \emph{For each $R>0$, there exist two constants $0<r_1<r_2\le R$
such that
$$
S(r_1,r_2):=\{x \in \R^d:\ r_1\le |x|\le r_2\}\subseteq \rm {supp}(\nu).
$$}
\end{itemize} It is obvious that {\bf{(A2)}} is weaker than {\bf{(A1)}}.

For any $\theta, c, r>0$, define
\begin{equation*}
\beta_{\theta,c}(r)={4\Phi_0\Big(\frac{r}{2}\Big)}{\Phi_1\big(e^{c(\Phi_0(\frac{r}{2}))^{\frac{1}{\theta}}}\big)},
\end{equation*}
where  \begin{equation}\label{symbol-1}\Phi_0(r)=(2\pi)^{-d}\int e^{-r|q(\xi)|}\,d\xi,\quad \Phi_1(r)=\sup_{|\xi|\le r}|q(\xi)|.\end{equation} We have the following statement for intrinsic ultracontractivity of $(T_t^D)_{t \ge 0}$ under {\bf({A2})} and
{\bf ({LDI})}.

\begin{theorem}\label{t1-2}  Suppose that $X$ is a symmetric L\'evy process such that {\bf{(A2)}} holds true, and that
{\bf {(LDI)}} also holds for the bounded domain $D$.
If there exists a constant $\theta>0$ such that for any $c>0$,
$$\Psi_{\theta,c}(r):=\int_r^\infty \frac{\beta_{\theta,c}^{-1}(s)}{s}\,ds<\infty, \quad r\ge 1,$$ then
the associated Dirichlet semigroup $(T_t^D)_{t \ge 0}$ is
intrinsically ultracontractive, and  there are constants $c_1,c_2>0$ such that for all $t>0$ and $x,y\in D,$
$$ p^D(t,x,y)\le  c_1\Psi^{-1}_{\theta,c_2}(t\wedge 1)e^{-\lambda_1t} \phi_1(x)\phi_1(y).$$
Here, we use
the convention that $f^{-1}(r)=\inf\{s>0: f(s)\le r\}$ and $\inf\emptyset=\infty.$    \end{theorem}

The intrinsic ultracontractivity for Dirichlet semigroup of symmetric $\alpha$-stable
process on a bounded H\"older domain of order $0$ was established in \cite{CS1}. Theorem \ref{t1-2} generalizes
such result to more general symmetric L\'evy process, whose L\'evy measure may be singular or may not satisfy {\bf (A1)}.
This can be seen from the following example.

\begin{example}\label{exm2} \it Let $X$ be a symmetric L\'{e}vy process with L\'{e}vy measure $\nu$ as follows
$$\nu(A)=\sum_{i=0}^{\infty}\int_{A}\frac{1}{|z|^{d+\alpha}}\I_{\{2^{-2i-1}\le |z|\le 2^{-2i}\}}\,dz,\quad A\in \mathscr{B}(\R^d)$$ for
some $\alpha\in(0,2)$.
Let $D$ be a bounded H\"older domain of order $0$. Then, the associated Dirichlet semigroup $(T_t^D)_{t \ge 0}$ is intrinsically ultracontractive, and
for every $\theta>{d}/{\alpha}$, there exist constants $c_1,c_2>0$ such that
$$p^D(t,x,y)\le  c_1e^{-\lambda_1 t}\exp\Big(c_2(1+t^{-\frac{d}{\alpha \theta-d}})\Big) \phi_1(x)\phi_1(y),\quad t>0, x,y\in D.
$$
\end{example}

\ \

The rest of this paper is arranged as follows. In Section
\ref{sec2} we present some preliminary results. Under assumptions {\bf (A1)} and {\bf (RC)}, we verify that
 for general L\'{e}vy process the Dirichlet heat kernel $p^D(t,x,y)$ is strictly positive for every $t>0$ and $x,y\in D$. In particular, Corollary \ref{c2-1} here also yields the strictly positivity of the transition density $p(t,x,y)$, which is interesting of its own.
In Section \ref{section3}, we prove Theorem \ref{t1-1} by making use of the methods in \cite{G,KS1,K} with some significant modifications.
The last section is devoted to the proof of Theorem \ref{t1-2}.
Comparing with the idea used in Section \ref{section3}, here we need establish the super Poincar\'{e} inequality for non-local Dirichlet forms and derive explicit lower bound for ground state in term of characteristic exponent.

\section{Preliminary Result: the Strict Positivity of Dirichlet Heat Kernel}\label{sec2}
The following lemma, similar to \cite[Lemma 2.5]{G}, is a direct consequence of Assumption {\bf (A1)}.
\begin{lemma}\label{l2-1}
Suppose {\bf {(A1)}} holds. Then for any $0<r<R_0$,
\begin{equation}\label{l2-1-1}
\delta(r):=\inf_{|x|\le R_0}\nu\big(B(x,r)\big)>0,
\end{equation}
where $R_0>0$ is the constant in
{\bf {(A1)}}.
\end{lemma}
\begin{proof}
Suppose that
\begin{equation*}
\inf_{|x|\le R_0}\nu\big(B(x,r_0)\big)=0
\end{equation*} for some $0<r_0<R_0$.
Then there exists a sequence $\{x_n\}_{n=0}^{\infty}\subseteq \overline{B(0,R_0)}$ such that
\begin{equation}\label{l2-1-2}
\lim_{n \rightarrow \infty} x_n=x_0
\end{equation} and
\begin{equation}\label{l2-1-3}
\lim_{n \rightarrow \infty}\nu\big(B(x_n,r_0)\big)=0.
\end{equation}
According to \eqref{l2-1-2} and \eqref{e1-1}, for $n$ large enough
\begin{equation*}
\nu\big(B(x_n,r_0))\ge \nu\Big(B\Big(x_0,\frac{r_0}{2}\Big)\Big)>0,
\end{equation*}
which contradicts with (\ref{l2-1-3}). This proves our desired conclusion
\eqref{l2-1-1}.
\end{proof}

 Next, we turn to the strictly positivity of the Dirichlet heat kernel $p^D(t,x,y)$. We first recall the parabolic property of the Dirichlet heat kernel $p^D(t,x,y)$ and the L\'{e}vy system of L\'{e}vy process $X$.

\begin{lemma}\label{l-p-1} \begin{itemize}
\item[(1)] The Dirichlet heat kernel $p^D(t,x,y)$ enjoys the parabolic property, i.e. for any $0<s<t$, $x,y \in D$ and stopping time $\tau$ with $\tau \le \tau_D$,
\begin{equation}\label{l2-2-1}
p^D(t,x,y)=\Ee^x\big[p^D(t-\tau\wedge s, X_{\tau\wedge s},y)\big].
\end{equation}
\item[(2)] Let $f$ be a non-negative measurable function on $\R_+\times \R^d\times \R^d$ vanishing on the diagonal. Then for every $x\in \R^d$ and stopping time $T$,
    \begin{equation}\label{l2-2-2}
    \Ee^x\bigg(\sum_{s\le T}f(s, X_{s-},X_s)\bigg)=\Ee^x \bigg[\int_0^T\int_{\R^d} f(s,X_s,X_s+z)\,\nu(dz)\,ds\bigg].
    \end{equation}
\end{itemize} \end{lemma}
\begin{proof}
 (1) We mainly follow the proof of \cite[Theorem 4.5]{CK} to prove the parabolic property for
$p^D(t,x,y)$. For fixed $t_0>0$ and $y \in D$, let
$q(s,x)=p^D(t_0-s,x,y)$ on $[0,t_0)\times D$.
For every
$(t,x)\in [0,t_0)\times D$, define a $\R_+ \times D$-valued process
$Y$ by $Y_s:=(t+s,X^D_s)$ for $0\le s < t_0-t$, and denote $\{\widetilde {\mathscr{F}}_s,\ 0\le s<t_0-t\}$ by the associated natural filtration. The law of the space-time process $s\mapsto Y_s$ starting from $(t,x)$ will be denoted by $\Pp^{(t,x)}$. Since for each $t>0$, $\sup_{x,y \in D}p^D(t,x,y)<\infty$,
$q(Y_s)$ is integrable for every $0\le s < t-t_0$. Then, for
every $0<r<s< t_0-t$,
\begin{equation*}
\begin{split}
 \Ee^{(t,x)}\left[q\left(Y_s\right)\Big|\widetilde {\mathscr{F}}_r\right]& =
\Ee^{x}\left[p^D\left(t_0-t-s,X_s^D,y\right)\Big| \mathscr{F}_r\right]\\
&=\Ee^{X_r^D}\left[p^D\left(t_0-t-s,X_{s-r}^D,y\right)\right]\\
&=\int_D p^D\left(s-r,X_r^D,z\right)p^D\left(t_0-t-s,z,y\right)\, dz\\
&=p^D\left(t_0-t-r,X_r^D,y\right)=q\left(Y_r\right),
\end{split}
\end{equation*}
where in the second equality $\{\mathscr{F}_s,\ 0\le s< t_0-t\}$ denotes the natural filtration generated by $X^D$ and we have used
the Markov property of $X^D$, and the fourth equality follows from the semigroup property of the Dirichlet heat kernel.
Hence, $\{q(Y_s),\widetilde {\mathscr{F}}_s,\  0\le s<t_0-t\}$  is a martingale.

For every $t>0$, choosing $t_0=t$ in the definition of $q$ above and using the optional sampling theorem, we find
for every $0<s<t$ and stopping time $\tau \le \tau_D$
\begin{equation*}
\begin{split}
& p^D(t,x,y)=q(0,x)=\Ee^{(0,x)}\left[q\left(Y_{s\wedge \tau}\right)\right]
=\Ee^x\left[p^D\left(t-s\wedge \tau, X^D_{s\wedge \tau},y\right)\right].
\end{split}
\end{equation*}
This finishes the proof of \eqref{l2-2-1}.

(2) We can follow the argument of \cite[Section 5]{CK1} (in particular \cite[(5.3)]{CK1}) to get \eqref{l2-2-2}, and
the details are omitted here.
\end{proof}

The main result of this section is the following
\begin{proposition}\label{p2-1}
Let $X$ be a (not necessarily symmetric) L\'{e}vy process satisfying ${(\bf{A1})}$, and let
$D$ be an open (not necessarily bounded) set such that ${(\bf{RC})}$ holds true. Then,
\begin{equation}\label{p2-1-1aa}
p^D(t,x,y)>0,\quad \forall\ t>0,x,y\in D.
\end{equation}
\end{proposition}

As a direct consequence of Proposition \ref{p2-1}, we have the following statement, which is interesting of its own.
\begin{corollary} \label{c2-1} Let $X$ be a (not necessarily symmetric) L\'{e}vy process satisfying ${(\bf{A1})}$. For any connected (not necessarily bounded) open set $D$,
$$p^D(t,x,y)>0,\quad \forall\ t>0,x,y\in D.$$ In particular,
\begin{equation*}
p(t,x,y)>0,\quad \forall\ t>0,x,y\in \R^d,
\end{equation*}
where $p(t,x,y)$ is the transition density for the process $X$.
\end{corollary}

\begin{proof}[Proof of Proposition $\ref{p2-1}$] The proof is split into three steps, and the first two steps are devoted to the proof of Corollary \ref{c2-1}.

(1) We show that for any connected open set $D$, $T_t^D(\I_U)(x)>0$ for every $x\in D$, connected open set $U\subseteq D$ and $t>0$.
According to
\cite[Theorem 5.1]{BSW}, there is a constant $c_0>0$ such that for every $r,t>0$ and $x \in \R^d$
\begin{equation*}\label{p2-1-2}
\Pp^x \big(\tau_{B(x,r)}>t\big)\ge 1 -c_0t\sup_{|\xi|\le \frac{1}{r}}|q(\xi)|.
\end{equation*}
In particular, for any $r>0$, we can find a constant $t(r)>0$ such that
\begin{equation*}
\Pp^x\big(\tau_{B(x,r)}>t(r)\big)\ge \frac{1}{2},\ \ \forall\ x \in \R^d.
\end{equation*}

Let $R_0$ be the constant in
Assumption {\bf (A1)}. Since $D$ is connected,
for every $x \in D$, connected open set $U$ and $t>0$,
there exist constants $\tilde t_1:=\tilde t_1(x,U,t)>0$, $0<
\tilde r_1:=\tilde r_1(x,U,t)<\frac{R_0}{8}$ and a sequence $\{x_i\}_{i=1}^{N+1}\subseteq  D$ with $N\ge \big[\frac{t}{\tilde t_1}\big]\ge 3$, such that the following properties hold:
\begin{itemize}
\item[(i)] for every $1\le i \le N$, $B_i \subseteq D$, $B_{N+1}\subseteq U$, and $B_i \bigcap B_{i+1}=\varnothing$, where $B_i:=B(x_i,2\tilde r_1)$ and $x_1=x$. (Note that we do not require that $B_i \bigcap B_{j}=\varnothing$ for any $i\neq j$, and so it may happen that $B_i=B_j$ for some $j\neq i+1$.)

\item[(ii)] for every $1\le i \le N$ and $y_i \in B_i$, $|y_i-y_{i+1}|\le \frac{R_0}{2}$.

\item[(iii)] for every $z \in \R^d$,
\begin{equation}\label{p2-1-2aa}
\Pp^z \big(\tau_{B(z,\tilde r_1)}>2\tilde t_1\big)\ge \frac{1}{2}.
\end{equation}
\end{itemize} Below, define a sequence of
stopping times $\{\tilde \tau_{B_{i}}\}_{i=1}^{N}$ as follows
\begin{equation*}
\begin{split}
&\tilde \tau_{B_0}=0,\ \tilde \tau_{B_1}:=\tau_{B_1},\,\, \tilde \tau_{B_{i+1}}:=\inf\{t>\tilde \tau_{B_i}: X_t \notin B_{i+1}\},\quad  1 \le i \le N-1,
\end{split}
\end{equation*} and let $\tilde B_i:=B(x_i,\tilde r_1)$ for $1 \le i \le N+1$.
Then, we have
\begin{equation}\label{p2-1-3a}
\begin{split}
&T_{t}^D(\I_U)(x)\\
&\ge T^D_{t}(\I_{B_{N+1}})(x)\\
&=\Ee^x\Big(\I_{B_{N+1}}(X^D_{t})\Big)\\
&\ge \Pp^x\Big(
\big(1-\frac{1}{N}\big)\tilde t_1<\tilde{\tau}_{B_{i}}-\tilde{\tau}_{B_{i-1}}<\tilde t_1 \ {\rm for\ each}\ 1\le i \le N,\\
&\qquad\quad\,\,\textrm{ and } \forall_{s\in[\tilde{\tau}_{B_{N}},t] } X^D_{s} \in B_{N+1}\Big)\\
&\ge \Pp^x\Big(
\big(1-\frac{1}{N}\big)\tilde t_1<\tilde{\tau}_{B_{i}}-\tilde{\tau}_{B_{i-1}}<\tilde t_1\textrm{ and }X_{\tilde{\tau}_{B_{i}}}\in \tilde B_{i+1}\ {\rm for\ each}\ 1\le i \le N,\\
&\qquad\quad\,\,\textrm{ and } \forall_{s\in[\tilde{\tau}_{B_{N}},t] } X^D_{s} \in B_{N+1}\Big)\\
& = \Pp^x\Big(\big(1-\frac{1}{N}\big)\tilde t_1<{\tau}_{B_{1}}<\tilde t_1, X_{{\tau}_{B_{1}}}\in
\tilde B_{2};\\
&\qquad\quad\,\,\cdot\Pp^{X_{\tilde{\tau}_{B_{1}}}}\Big(\big(1-\frac{1}{N}\big)\tilde t_1<\tau_{B_{2}}
<\tilde t_1,X_{{\tau}_{B_{2}}}\in
\tilde B_{3};\\
&\qquad\quad\,\,\cdot\Pp^{X_{\tilde{\tau}_{B_{2}}}}\Big(\cdots \Pp^{X_{\tilde{\tau}_{B_{N-1}}}}\Big(\big(1-\frac{1}{N}\big)\tilde t_1<{\tau}_{B_{N}}<
\tilde t_1,X_{\tau_{B_{N}}}\in
\tilde B_{N+1};\\
&\qquad\quad\,\,\cdot
\Pp^{X_{\tilde{\tau}_{B_{N}}}}\Big(\forall_{s\in[0,t-\tilde {\tau}_{B_N}]} X_{s}
\in B_{N+1}\Big)\Big)\cdots\Big)\Big)\Big),
\end{split}
\end{equation}
where in the last equality we used the strong Markov property.

Note that, if for any $1\le i\le N$, $$\big(1-\frac{1}{N}\big)\tilde t_1<\tilde{\tau}_{B_{i}}-\tilde{\tau}_{B_{i-1}}<\tilde t_1,$$
then
\begin{equation*}
\begin{split}
t-\tilde \tau_{B_N}&\le t-N\big(1-\frac{1}{N}\big)\tilde t_1
= t-N\tilde t_1+ \tilde t_1 \le 2 \tilde t_1,
\end{split}
\end{equation*}
where the last inequality follows from the fact $t-N\tilde t_1\le \tilde t_1$.
Thus, when $X_{\tilde{\tau}_{B_{N}}}\in \tilde B_{N+1}$ and
$(1-\frac{1}{N})\tilde t_1<\tilde{\tau}_{B_{i}}-\tilde{\tau}_{B_{i-1}}<\tilde t_1$ for all $1 \le i \le N$, we have
\begin{equation*}
\begin{split}\Pp^{X_{\tilde{\tau}_{B_{N}}}}\Big(\forall_{s\in[0, t-\tilde {\tau}_{B_N}]} X_{s}
\in B_{N+1}\Big)&
\ge \inf_{y\in \tilde{B}_{N+1}}\Pp^y\Big(X_t\in B(y, \tilde r_1) \textrm{ for all }0 < t \le 2\tilde t_1\Big)\\
&\ge \inf_{y\in \tilde{B}_{N+1}}\Pp^y\big(\tau_{B(y,\tilde r_1)}>2\tilde  t_1\big)\ge \frac{1}{2}.
\end{split}
\end{equation*}
where the last inequality we used \eqref{p2-1-2aa}.

On the other hand,  for any
$1\le i\le N$, if $X_{\tilde{\tau}_{B_{i-1}}} \in \tilde B_{i}$, then, according to the L\'{e}vy system of the process $X$
(see Lemma \ref{l-p-1}),
\begin{align*}
&\Pp^{X_{\tilde{\tau}_{B_{i-1}}}}\Big(\big(1-\frac{1}{N}\big)\tilde t_1<{\tau}_{B_{i}}<\tilde t_1,X_{{\tau}_{B_{i}}}\in
\tilde B_{i+1}\Big)\\
&\ge \inf_{y \in \tilde B_i}\int_{(1-\frac{1}{N})\tilde t_1}^{\tilde t_1}\int_{B_i}p^{B_i}(s,y,z)\Big(
\int_{\tilde B_{i+1}-z}\nu(dw)\Big)\,dz\,ds\\
&\ge \frac{\tilde t_1}{N}\Big(\inf_{y \in \tilde B_i}\Pp^y(\tau_{B_i}>\tilde t_1)
\Big)\Big(\inf_{z \in B_i}\nu\big(B(x_{i+1}-z,\tilde r_1)\big)\Big)\\
&\ge \frac{\tilde t_1}{N}\Big(\inf_{y \in \tilde B_i}\Pp^y(\tau_{B(y,\tilde r_1)}>\tilde t_1)
\Big)\Big(\inf_{z \in B_i}\nu\big(B(x_{i+1}-z,\tilde r_1)\big)\Big).
\end{align*}
By \eqref{p2-1-2aa}, $$\inf_{y \in \tilde B_i}\Pp^y(\tau_{B(y,\tilde r_1)}>\tilde t_1)>\frac{1}{2}.$$ For every $z \in B_i$, since
$|x_{i+1}-z|\le \frac{R_0}{2}$ and $\tilde r_1<\frac{R_0}{2}$, $B(x_{i+1}-z,\tilde r_1)\subseteq B(0,R_0)$. Then, Assumption
{\bf {(A1)}} and Lemma \ref{l2-1} yield that $$\inf_{z \in B_i}\nu\big(B(x_{i+1}-z,\tilde r_1)\big)>0.$$ Therefore,  for any
$1\le i\le N$ and  $X_{\tilde{\tau}_{B_{i-1}}} \in \tilde B_{i}$, $$\Pp^{X_{\tilde{\tau}_{B_{i-1}}}}\Big(\big(1-\frac{1}{N}\big)\tilde t_1<{\tau}_{B_{i}}<\tilde t_1,X_{{\tau}_{B_{i}}}\in
\tilde B_{i+1}\Big)>0.$$

Combining all the estimates above with \eqref{p2-1-3a}, we obtain that $T_t^D(\I_U)(x)>0$.

(2) For any connected open set $D$, we have proved that $T_t^D(\I_U)(x)>0$ for any $x \in D$, $t>0$ and open connected subset $U\subseteq D$. So,
$p^D(t,x,z)>0$ for almost surely $z \in D$ with respect to the Lebesgue measure (the exceptional set may depend on
$x\in D$ and $t>0$). Furthermore, it is obvious that if Assumption {\bf {(A1)}} holds for $\nu$, then it also holds for the L\'evy measure
$\hat \nu$ of the dual process $\hat X$. Then, following the arguments in step (1), we can obtain that for every
$x\in D$ and $t>0$, $\hat p^D(t,x,z)>0$ for almost surely $z \in D$.

Assume that $p^D(t,x,y)=0$ for some $x,y \in D$ and $t>0$. Then,
\begin{equation}\label{p2-1-4a}
\begin{split}
0=p^D(t,x,y)&=\int_D p^D\big(\frac{t}{2},x,z\big)p^D\big(\frac{t}{2},z,y\big)\,dz\\
&=\int_D p^D\big(\frac{t}{2},x,z\big)\hat p^D\big(\frac{t}{2},y,z\big)\,dz.
\end{split}
\end{equation}
On the other hand, according to the conclusions above, $p^D(\frac{t}{2},x,z)\hat p^D(\frac{t}{2},y,z)>0$
for almost surely $z \in D$, which is a contradiction with \eqref{p2-1-4a}. Therefore, the assumption above is not true; that is,
$p^D(t,x,y)>0$ for every $x,y \in D$ and $t>0$.

(3) Now we consider an open set satisfying {\bf {(RC)}}. It is easy to see that in this case for every $x,y \in D$, there exist an integer
$m\ge 1$, some constants $0<\varepsilon<1$ and $0<r_0<\frac{\varepsilon R_0}{4}$ (here $R_0$ is the constant in {\bf {(A1)}}) and points $x_j \in D$ for $1 \le j \le m$, such that
\begin{itemize}
\item[(i)] $x \in B(x_1,r_0)$, $y \in B(x_m,r_0)$.

\item[(ii)] for every $1\le j \le m-1$,
$|x_j-x_{j+1}|\le (1-\varepsilon)R_0$.

\item[(iii)] for every $1\le i, j\le m$ with $i \neq j$, $K_i \bigcap K_j=\varnothing$, where $K_i:=B(x_i,r_0)$.
\end{itemize}

For every $t>0$, $1 \le j \le m-1$ and $z_j \in K_j$, by the parabolic property of Dirichlet heat kernel $p^D(t,x,y)$ and the L\'{e}vy system of the process $X$, see Lemma \ref{l-p-1},
\begin{equation}\label{p2-1-3}
\begin{split}
&p^D(t,z_{j},z_{j+1})\\
&=\Ee^{z_{j}}\Big[p^D\big(t- \frac{t}{2}\wedge \tau_{K_{j}},X_{\frac{t}{2}\wedge\tau_{K_j}},z_{j+1}\big)\Big]\\
&\ge \Ee^{z_{j}}\Big[p^D\big(t- \tau_{K_{j}},X_{\tau_{K_j}},z_{j+1}\big)\I_{\{\tau_{K_j}\le \frac{t}{2}\}}
\I_{\{X_{\tau_{K_j}}\in K_{j+1}\}}\Big]\\
&=\int_0^{\frac{t}{2}} \int_{K_j} p^{K_j}(s,z_j,z)\int_{K_{j+1}-z}p^D(t-s,z+\tilde z,z_{j+1})\,\nu(d\tilde z)\,dz\,ds\\
&\ge \int_{\frac{t}{4}}^{\frac{t}{2}} \int_{\tilde K_j} p^{K_j}(s,z_j,z)\int_{B(x_{j+1}-x_j,r_0/2)}p^D(t-s,z+\tilde z,z_{j+1})\,\nu(d\tilde z)\,dz\,ds,
\end{split}
\end{equation}
where the last inequality follows from the fact that $B(x_{j+1}-x_j,r_0/2)\subseteq K_{j+1}-z$ for any
$z \in \tilde K_j:=K_j/2=B(x_j,r_0/2)$.

By the conclusion in step (2), for every connected set $U \subseteq D$,
\begin{equation}\label{p2-1-2a}
p^D(t,x,y)\ge p^{U}(t,x,y)>0,\quad \forall\ t>0,\ x,y\in U.
\end{equation}
According to (\ref{p2-1-2a}) and the fact that for every $t>0$, $p^D(t,\cdot,\cdot):D\times D \rightarrow [0,\infty)$ is continuous,
we know that
\begin{equation}\label{p2-1-3-4}\begin{split}\inf_{ z\in \tilde K_j, \tilde{z}\in B(x_{j+1}-x_j,r_0/2)} & p^D(t-s,z+\tilde z,z_{j+1})\\
&\ge \inf_{z\in \tilde K_j,
\tilde{z}\in K_{j+1}-z}  p^D(t-s,z+\tilde z,z_{j+1})\\
&\ge \inf_{z\in  K_{j+1}}  p^D(t-s,z,z_{j+1})\\
&=: C(t-s, r_0, x_{j+1}, z_{j+1})>0.\end{split}\end{equation}

Next, we suppose that $p^D(t,z_j,z_{j+1})=0$ for some $t>0$, $z_j \in K_j$ and $1\le j\le m-1$. Then,
by (\ref{p2-1-3}) and \eqref{p2-1-3-4},
\begin{equation*}
 \int_{\frac{t}{4}}^{\frac{t}{2}} C(t-s, r_0, x_{j+1}, z_{j+1}) \int_{\tilde K_j} p^{K_j}(s,z_j,z)\int_{B(x_{j+1}-x_j,r_0/2)}\,\nu(d\tilde z)\,dz\,ds=0,
\end{equation*}
which, along with \eqref{p2-1-3-4}, {\bf{(A1)}} and the fact that $B(x_{j+1}-x_j,r_0/2)\subseteq B(0,R_0)$ due to
$|x_{j+1}-x_j|\le (1-\varepsilon)R_0$ and $r_0<\frac{\varepsilon R_0}{4}$, in turn implies that
\begin{equation}\label{p2-1-4}
p^{K_j}(s,z_j,z)=0
\end{equation} holds for $(s,z)\in \big[\frac{t}{4},\frac{t}{2}\big]
\times\tilde K_j$ almost surely under the measure $ds\,dz$.
However,  according to \eqref{p2-1-2a}, for
every $s>0$ and $\tilde x, \tilde y \in K_j$
\begin{equation}\label{p2-1-6}
p^{K_j}(s,\tilde x, \tilde y)>0.
\end{equation}
This is a contradiction with (\ref{p2-1-4}), whence
\begin{equation}\label{p2-1-5}
p^D(t,z_j, z_{j+1})>0,\quad \forall\, t>0, z_j \in K_j,1\le j\le m-1.
\end{equation}

Finally, for every $t>0$ and $x,y\in D$,
\begin{align*}
p^D(t,x,y)
&=\int_D \dots \int_D p^D\big(\frac{t}{m},x,z_1\big)p^D\big(\frac{t}{m},z_1,z_2\big)\dots
p^D\big(\frac{t}{m},z_m,y\big)\,dz_1\dots \,dz_m\\
&\ge\int_{K_1}\dots \int_{K_m} p^D\big(\frac{t}{m},x,z_1\big)p^D\big(\frac{t}{m},z_1,z_2\big)\dots
p^D\big(\frac{t}{m},z_m,y\big)\,dz_1\dots \,dz_m\\
&\ge \int_{K_1} \dots \int_{K_m} p^{K_1}\big(\frac{t}{m},x,z_1\big)p^D\big(\frac{t}{m},z_1,z_2\big)\dots\\
&\qquad\qquad\qquad\times p^{D}\big(\frac{t}{m},z_{m-1},z_m\big)p^{K_m}\big(\frac{t}{m},z_m,y\big)\,dz_1\dots \,dz_m.
\end{align*}
This along with \eqref{p2-1-6} and (\ref{p2-1-5}) gives us that $p^D(t,x,y)>0$ for every $x,y \in D$ and $t>0$, which proves our desired assertion.
\end{proof}

We conclude with two remarks on Proposition \ref{p2-1} and Corollary \ref{c2-1}.

\begin{remark}
(1) When L\'{e}vy process $X$ is symmetric and $D$ is a bounded connected open set, the strict positivity of Dirichlet heat kernel $p^D(t,x,y)$ was proved in \cite[Proposition 2.2(i)]{G} without any additional condition on the L\'{e}vy measure. However, the proof heavily depends on the symmetric property, and it does not work for Corollary \ref{c2-1}. An interesting point for Corollary \ref{c2-1} is due to that it is concerned about non-symmetric L\'{e}vy processes. Based on Proposition \ref{p2-1}, some arguments for examples in \cite[Section 4]{KS1} can be shortened. Furthermore, according to the proofs of Proposition \ref{p2-1} and Lemma
\ref{l4-2} below (in particular, see the construction of a sequence of subsets $\{D_i\}_{i=1}^n$ here),
we can verify that, under the weaker assumption {\bf(A2)} on the L\'evy measure $\nu$, for any connected (not necessarily bounded) open set $D$,
$p^D(t,x,y)>0$ for any $t>0$ and $x,y\in D.$ The details are left to readers.

(2) The proof of Proposition \ref{p2-1} is only based on the probability estimate of the first exit time and the L\'evy system of
L\'{e}vy process $X$, both of which are available for general L\'{e}vy type processes, see e.g.\ \cite{BSW,CK1}. Therefore, Proposition \ref{p2-1} and so Corollary \ref{c2-1} still hold true for a large class of L\'{e}vy type jump processes.
\end{remark}

\section{Proof of Theorem \ref{t1-1}}\label{section3}
Throughout this section, we always assume that assumption {\bf (A1)} holds true, and the ground state $\phi_1$ and its dual ground state $\hat\phi_1$ are bounded, continuous and strictly positive. To prove Theorem \ref{t1-1}, we mainly use the methods in \cite{G,KS1,K} but with non-trivial modifications.
Since $D$ is a bounded set, there exist finite open subsets
$\{\tilde D_i\}_{i=1}^n$ such that
\begin{itemize}
\item[(i)] $D=\bigcup_{i=1}^n \tilde D_i$.

\item[(ii)] for any $1\le i \le n$ and $\tilde x_i, \tilde y_i$ $\in \tilde D_i$, we have
$|\tilde x_i-\tilde y_i|\le \frac{R_0}{2}$.

\item[(iii)] there are
$0<r_0<\frac{R_0}{8}$ and finite points $\{x_i\}_{i=1}^n$ such that
$B(x_i,2r_0)$ $\subseteq \overline{B(x_i,2r_0)}\subseteq \tilde D_i$ for every
$1 \le i \le n$.

\end{itemize}
Below, we define
\begin{equation*}\label{e3-1}
A:=\bigcup_{i=1}^n B(x_i,\frac{r_0}{2}),\ \
B:=\bigcup_{i=1}^n \overline{B(x_i, r_0)},\ \
C:=\bigcup_{i=1}^n B(x_i,2r_0).
\end{equation*}

For every open set $U \subseteq \R^d$, let
\begin{equation*}
G^U(x,y)=\int_0^{\infty} p^U(t,x,y)\,dt, \,\,
\hat G^U(x,y)=\int_0^{\infty} \hat p^U(t,x,y)\,dt,\quad \forall\ x,y\in U
\end{equation*}
be the Green functions for the Dirichlet semigroup $(T^U_t)_{t \ge 0}$ and
$(\hat T^U_t)_{t \ge 0}$ respectively, e.g.\ see \cite{KS1}. Define
 $$\eta_U=\inf\{t \ge 0, X_t \notin U\},\quad \hat\eta_U=\inf\{t \ge 0, \hat X_t \notin U\}.$$
We first provide the following estimate, which is crucial for the proof of Theorem \ref{t1-1}.
\begin{lemma}\label{l3-1}
There exists a constant $c_1>0$ such that
for every $x \in \R^d$,
\begin{equation}\label{l3-1-1}
\Pp^x\big(X_{\eta_{D \setminus B}}\in B\big)\ge c_1\Ee^x\big[\eta_{D \setminus B}\big],\quad
\Pp^x\big(\hat X_{\eta_{D \setminus B}}\in B\big)\ge c_1\Ee^x\big[\hat \eta_{D \setminus B}\big].
\end{equation}
\end{lemma}
\begin{proof}
For every $x \notin D$ or $x \in B$, we have $\Pp^x\big(\eta_{D \setminus B}=0\big)=1$, which immediately implies that the estimate for $X$ in
(\ref{l3-1-1}) holds true. Now we assume $x \in D\setminus B$, and so
$\eta_{D \setminus B}=\tau_{D \setminus B}$, $\Pp^x$-a.s.. By the Ikeda-Watanabe formula, see \cite[(2.1)]{G},
\begin{equation}\label{l3-1-2}
\begin{split}
\Pp^x\big(X_{\eta_{D \setminus B}}\in B\big)&
\ge \Pp^x\big(X_{\tau_{D \setminus B}}\in A\big)=\int_{D \setminus B} G^{D \setminus B}(x,y)\int_{A-y}\nu(dz)\,dy.
\end{split}
\end{equation}
For every $y \in D\setminus B$, there exists an integer $1 \le i \le n$ such that
$y \in \tilde D_i$. Then, by the definition of $A$, we obtain
\begin{equation}\label{l3-1-3}
B\big(x_i-y,\frac{r_0}{4}\big)\subseteq A-y.
\end{equation}
Moreover, since $y,x_i\in \tilde D_i$, by the property of $\tilde D_i$ we know that
$|x_i-y|\le \frac{R_0}{2}$. Combining (\ref{l3-1-3}) with (\ref{l2-1-1}) yields that
for every $y \in D\setminus B$
\begin{equation}\label{l3-1-4}
\nu(A-y)\ge \nu\big(B(x_i-y,\frac{r_0}{4})\big)\ge \delta(\frac{r_0}{4})>0.
\end{equation}
According to (\ref{l3-1-4}) and (\ref{l3-1-2}),
\begin{equation*}
\begin{split}
\Pp^x\big(X_{\eta_{D \setminus B}}\in B\big)&
\ge \delta(\frac{r_0}{4})\int_{D \setminus B} G^{D \setminus B}(x,y)\,dy=\delta(\frac{r_0}{4})\Ee^x\big[\tau_{D \setminus B}\big]
=\delta(\frac{r_0}{4})\Ee^x\big[\eta_{D \setminus B}\big],
\end{split}
\end{equation*}
which arrived at the first desired assertion in (\ref{l3-1-1}) with $c_1=\delta(\frac{r_0}{4})$.

Following the arguments above, we can also obtain the estimate in \eqref{l3-1-1} for the dual process $\hat X$.
\end{proof}

\begin{lemma}\label{l3-2}
There exists a constant $c_2>0$ such that
for every $x \in D$
\begin{equation}\label{l3-2-1}
\begin{split}&\int_{C}G^D(x,y)\,dy \ge c_2 \int_{D \setminus C}G^D(x,y)\,dy,\\
&\int_{C}\hat G^D(x,y)\,dy \ge c_2 \int_{D \setminus C}\hat G^D(x,y)\,dy.
\end{split}\end{equation}
\end{lemma}

\begin{proof} The proof is mainly based on Lemma \ref{l3-1} and the argument of \cite[Lemma 3.5]{KS1} (see also \cite{G,K}). We present the sketch here for the sake of completeness. It suffices to show the first estimate in
\eqref{l3-2-1}, since the second one for the dual process $\hat X$ can be
proved similarly.

Let $\theta_t$ denote the $t$-time shift operator for the process $X$. Define a sequence of stopping times
as follows
\begin{equation*}
S_1:=0,\ \ T_k:=S_k+\eta_{D \setminus B}\circ \theta_{S_k},\ \ S_{k+1}:=T_k+\eta_{C}\circ \theta_{T_k}
,\ \ k\ge 1.
\end{equation*}
According to (\ref{l3-1-1}) and the strong Markov property, we immediately have that for every $x\in\R^d$ and $k\ge 1$,
\begin{equation}\label{l3-2-2}
\Pp^x\big(X_{T_k} \in B\big)\ge c_1\Ee^x\big[T_k-S_k\big].
\end{equation}
By \cite[Lemma 3.4]{KS1},
\begin{equation}\label{l3-2-2a}
\begin{split}
& \lim_{k \rightarrow \infty}T_k=\lim_{k \to \infty}S_k=\tau_D,\quad \Pp^x\ a.s..
\end{split}
\end{equation}
Therefore we have
\begin{equation}\label{l3-2-3}
\begin{split}
\int_C G^D(x,y)\,dy&=\Ee^x\Big[\int_0^{\tau_D}\I_C(X_t)\,dt\Big]\\
&=\Ee^x\Big[\sum_{k=1}^{\infty}\Big(\int_{S_k}^{T_k}\I_C(X_t)\,dt+\int_{T_k}^{S_{k+1}}\I_C(X_t)\,dt\Big)\Big]\\
&\ge  \Ee^x\Big[\sum_{k=1}^{\infty}\int_{T_k}^{S_{k+1}}\I_C(X_t)\,dt\Big]\\
&=\sum_{k=1}^{\infty}\Ee^x\big[S_{k+1}-T_k\big],
\end{split}
\end{equation}
where the first step follows from the relation
$X_t=X_t^D$ for every $t<\tau_D$, and in the last step we have used the fact that $X_t \in C$ for every $T_k<t<S_{k+1}$.

It is well known that L\'{e}vy process enjoys the Feller property, i.e. its semigroup $T_t$ maps $C_\infty(\R^d)$ into $C_\infty(\R^d)$ for every $t>0$. By the separation property of Feller process,
\begin{equation*}
\inf_{y \in B}\Ee^y \tau_C\ge \frac{t_0}{2}
\end{equation*}
for some constant $t_0>0$, see \cite[(3.2)]{KS1}. Hence, due to the strong Markov property again, for every $x\in\R^d$ and $k \ge 1$,
\begin{equation*}
\begin{split}
\Ee^x\big[S_{k+1}-T_k\big]&= \Ee^x\Big[\Ee^{X_{T_k}}[\tau_C];T_k<\tau_D\Big]\\
&\ge \Pp^x(X_{T_k}\in B)\inf_{y \in B}\Ee^y\tau_C\\
&\ge \frac{c_1t_0}{2}\Ee^x\big[T_k-S_k\big]\\
&=:c\Ee^x\big[T_k-S_k\big],
\end{split}
\end{equation*}
where the last inequality follows from (\ref{l3-2-2}). Combining this estimate with (\ref{l3-2-3}) yields that
\begin{align*}
\int_C G^D(x,y)\,dy& \ge c\sum_{k=1}^{\infty}\Ee^x\big[T_k-S_k\big]\\
&\ge c\Ee^x\Big[\sum_{k=1}^{\infty}\int_{S_k}^{T_k}\I_{D\setminus C}(X_t)\,dt\Big]\\
&=c\Ee^x\Big[\sum_{k=1}^{\infty}\Big(\int_{S_k}^{T_{k}}\I_{D\setminus C}(X_t)\,dt+
\int_{T_k}^{S_{k+1}}\I_{D\setminus C}(X_t)\,dt\Big)\Big]\\
&=c\Ee^x\Big[\int_{0}^{\tau_D}\I_{D\setminus C}(X_t)\,dt\Big]\\
&=c\int_{D\setminus C}G^D(x,y)\,dy,\end{align*}
where in the forth step we have used again the fact that $X_t \in C$ for every $T_k<t<S_{k+1}$.
This proves the desired conclusion.
\end{proof}
According to Lemma \ref{l3-2}, we can give lower bound estimates for ground state $\phi_1$ and dual ground state $\hat\phi_1$.
\begin{lemma}\label{l3-3}
There exists a constant $c_3>0$, such that for every $x \in D$,
\begin{equation}\label{l3-3-1}
\Ee^x[\tau_D]\le c_3\phi_1(x),\quad \Ee^x[\hat \tau_D]\le c_3\hat \phi_1(x).
\end{equation}
\end{lemma}
\begin{proof}
We only verify the first estimate in \eqref{l3-3-1} here. By (\ref{l3-2-1}) we have for every $x \in D$
\begin{equation}\label{l3-3-2}
\begin{split}
&\Ee^x[\tau_D]=\int_{C} G^D(x,y)\,dy+\int_{D \setminus C}G^D(x,y)\,dy\le
\Big(1+\frac{1}{c_2}\Big)\int_{C}G^D(x,y)\,dy.
\end{split}
\end{equation}
Since $C$ is a precompact subset of $D$ and $\phi_1$ is strictly positive and continuous on $D$,
there is a constant $C_1>0$ such that $\inf_{z \in C}\phi_1(z)\ge C_1$. Hence for every
$x \in D$
\begin{equation*}
\begin{split}
\int_{C}G^D(x,y)\,dy &\le \frac{1}{C_1}\int_C G^D(x,y)\phi_1(y)\,dy\\
&\le \frac{1}{C_1}\int_D G^D(x,y)\phi_1(y)\,dy=\frac{1}{C_1\lambda_1} \phi_1(x),
\end{split}
\end{equation*} where in the equality we have used the fact that $\frac{\phi_1(x)}{\lambda_1}=\int_D G^D(x,y)\phi_1(y)\,dy$, see e.g.\ \cite{CS}.
Combining this with (\ref{l3-3-2}), we arrive at the conclusion (\ref{l3-3-1}).
\end{proof}

Now, we are in a position to present the
\begin{proof}[Proof of Theorem $\ref{t1-1}$]
According to (\ref{l3-3-1}), for any $t>0$, $x,y\in D$,
\begin{equation}\label{t1-1-1}
\begin{split}
p^D(t,x,y)&=\int_D p^D\big(\frac{t}{3},x,z\big)\int_D
p^D\big(\frac{t}{3},z,w\big)p^D\big(\frac{t}{3},w,y\big)\,dw\,dz\\
&\le c\big(\frac{t}{3}\big)\Big(\int_D p^D\big(\frac{t}{3},x,z\big)\,dz\Big)\Big(\int_D \hat p^D\big(\frac{t}{3},y,w\big)\,dw\Big)\\
&= c\big(\frac{t}{3}\big)\Pp^x\big(\tau_D>\frac{t}{3}\big)\Pp^y\big(\hat \tau_D>\frac{t}{3}\big)\\
&\le \frac{9c(\frac{t}{3})}{t^2}\Ee^x[\tau_D]\Ee^y[\hat \tau_D]\\
&\le \frac{9c_3^2c(\frac{t}{3})}{t^2}\phi_1(x)\hat \phi_1(y),
\end{split}
\end{equation}
where in the first inequality we have used the facts that
$p^D(t,w,y)=\hat p^D(t,y,w)$ and $\sup_{z,w\in D}p^D\big(\frac{t}{3},z,w\big)\le c\big(\frac{t}{3}\big)$,
and the second inequality follows from the Chebyshev inequality.
Hence, from (\ref{t1-1-1}) we know that the semigroup $(T_t^D)_{t \ge 0}$
is intrinsically ultracontractive.

Furthermore, according to \cite{KSc2}, we know that for every $t>0$,
$$c(t)\le (2\pi)^{-d}\int e^{-t \mathrm{Re}\,q(\xi)}\,d\xi,$$ which together with \eqref{t1-1-1} yields the desired assertion \eqref{t1-1-1-1} for $t>0$ small enough.
The estimate in \eqref{t1-1-1-1} for large $t$ follows from \cite[Theorem 2.7]{KS2}. By now we have finished the proof.
\end{proof}

To show the power of Theorem \ref{t1-1}, we take the following example about the truncated strictly $\alpha$-stable process.
In particular, comparing with \cite[Example 4.5]{KS1}, we do not require that $D$ is $\kappa$-fat.

\begin{example}\label{ex1-1} \it
Let $X$ be a L\'{e}vy process on $\R^d$ with L\'{e}vy measure as follows
\begin{equation}\label{ex1-1-1}
\nu(A)\ge c_0\int_{\Ss}\int_0^{r_0} \I_{A}(\theta r)\frac{1}{r^{d+\alpha}}\,dr \,\mu(d\theta),\quad A\in \mathscr{B}(\R^d),
\end{equation}
where $\alpha\in(0,2)$, $r_0,c_0>0$ and $\mu$ is a finite non-degenerate (not necessarily symmetric) measure
on the unit sphere $\Ss$ in the sense that its support is not contained in any proper linear subspace of $\R^d$.
Let $D$ be a bounded open set satisfying assumption {\bf(RC)} in Section $\ref{sec2}$ with $R_0$ to be the constant $r_0$ in \eqref{ex1-1-1}.
Then, the associated Dirichlet semigroup $(T_t^D)_{t \ge 0}$ is intrinsically ultracontractive, and for all $t>0$ and $x,y\in D$,
\begin{equation*}\label{ex1-1-2}
p^D(t,x,y)\le c_1 e^{-\lambda_1 t}\big(1+t^{-2-{d}/{\alpha}}\big)\phi_1(x)\hat\phi_1(y)
\end{equation*}
holds for some constant $c_1>0$.
\end{example}

\begin{proof}
Let $\nu$ be the L\'{e}vy measure given by \eqref{ex1-1-1}, and let $D$ be the open set satisfying {\bf ({RC})}.  According to Proposition \ref{p2-1}, Corollary \ref{c2-1} and (the proof of) \cite[Example 1.5]{SSW}, both the transition density $p(t,x,y)$ and the Dirichlet heat kernel $p^D(t,x,y)$ exist and fulfill all the conditions in Subsection \ref{section11}. It is obvious that {\bf{(A1)}} holds. Then, the desired assertion follows from Theorem \ref{t1-1}.
\end{proof}

\begin{remark} (1) As mentioned in the beginning of Subsection \ref{section13}, the intrinsic ultracontractivity of Dirichlet semigroups for Brownian motion on a bounded domain $D$ depends on the geometry of the boundary of $D$. Furthermore, by \cite[Theorem 1.1]{OW} and the conclusion of Example \ref{ex1-1}, quantitative estimates about $C(t)$ in \eqref{e1-3} are also different for intrinsically contractive Dirichlet semigroups between Brownian motion and L\'{e}vy jump process on bounded Lipschitz domains.

(2) The conclusion \eqref{t1-1-1-1} can apply to get explicit upper estimates for Dirichlt heat kernel $p^D(t,x,y)$. For instance, consider symmetric $\alpha$-stable process on bounded $\kappa$-fat domain $D$. Then, there is a constant $c>0$ such that for any $x,y\in D$ and $t\in(0,1]$,
$$p^D(t,x,y)\le c \left(t^{-d/\alpha}\wedge \frac{ t}{|x-y|^{d+\alpha}}\right) \left(\frac{\phi_1(x)}{t} \wedge1\right) \left(\frac{\phi_1(y)}{t} \wedge1\right).$$

Indeed, it was shown in \cite[Theorem 1]{BGR} that
for any $x,y\in D$ and $t\in(0,1]$,
\begin{equation}\label{eee} p^D(t,x,y)\asymp p(t,x,y)\Pp^x(\tau_D>t)\Pp^y(\tau_D>t),\end{equation} where  $p(t,x,y)$ is the transition density of $\alpha$-symmetric stable process, i.e.\
$$p(t,x,y) \asymp  \left(t^{-d/\alpha}\wedge \frac{ t}{|x-y|^{d+\alpha}}\right).$$
On the other hand, according to \eqref{t1-1-1-1}, there is a constant $c_1>0$ such that \begin{equation}\label{eee-1} p^D(t,x,y)\le  c_1t^{-d/\alpha-2} \phi_1(x)\phi_1(y),\quad x,y\in D, t\in(0,1].\end{equation} By \eqref{eee} and \eqref{eee-1}, we find that for some $c_2>0$,
$$\Pp^x(\tau_D>t)\le c_2\left(\frac{\phi_1(x)}{t} \wedge1\right),\quad x\in D,$$ which along with \eqref{eee} in turn yields the desired assertion.
\end{remark}

\section{Proof of Theorem \ref{t1-2}}\label{section4}
The main tool to prove Theorem \ref{t1-2} is different from that of Theorem \ref{t1-1}, and it is based on the (intrinsic) super Poincar\'e inequality for non-local Dirichlet forms (this is the reason why we need require $X$ to be symmetric in the section).
The super Poincar\'{e} inequality can be viewed as an alternative of Rosen's Lemma, which is in the context of the super log-Sobolev inequality, see e.g. \cite[Theorem 5.1]{DS}.

First, we recall some facts about Dirichlet form in our setting. Let $X$ be a symmetric L\'{e}vy process, and $D$ be a bounded domain. Then, the symmetric Dirichlet form $(\E^D,\D(\E^D))$ for the Dirichlet semigroup $(T^D_t)_{t\ge0}$ on $L^2(D;dx)$ is given by
\begin{equation*}
\begin{split}
\E^D(f,f)&=\int_{\R^d}\int_{\R^d}\big(f(x+z)-f(x)\big)^2\,\nu(dz)\, dx,\\
\D(\E^D)&=\overline{C_c^\infty(D)}^{\E_1^D},
\end{split}
\end{equation*}
where $C_c^\infty(D)$ is the set of $C^\infty$ functions on $D$ with compact support, and
$\overline{C_c^\infty(D)}^{\E_1^D}$ denotes the extension of $C_c^\infty(D)$ under the norm
$\|f\|_{\E^D_1}:=\sqrt{\E^D(f,f)+\|f\|_{L^2(D;dx)}^2}$.

Since $D$ is connected, the Dirichlet heat kernel $p^D(t,x,y)$ is strictly positive for every $t>0$ and $x,y\in D$, e.g.\ see \cite[Proposition 2.2(i)]{G}, and the associated ground state $\phi_1$ (corresponding to the first eigenvalue $\lambda_1$) can be chosen to be bounded, continuous and strictly positive.
The following result is essentially taken from \cite[Theorem 2.1 and Proposition 2.3]{OW}, which give us
sufficient conditions for intrinsic ultracontractivity of $(T_t^D)_{t \ge 0}$ in terms of (intrinsic) super Poincar\'e inequality for $(\E^D, \D(\E^D))$ and lower bound of ground state for $\phi_1$.

\begin{lemma}\label{l4-1}
Assume that there is a decreasing function $\beta_0:(0,\infty)\to(0,\infty)$ such that
\begin{equation}\label{l4-1-1}
\begin{split}
\int_D f^2(x)\,dx\le s  \E^D(f,f)+\beta_0(s)\Big(\int_D |f(x)|\,dx\Big)^2,\quad f \in C_c^\infty(D),\ s>0.
\end{split}
\end{equation}
Then the following intrinsic super Poincar\'{e} inequality holds
\begin{equation*}
\int_D f^2(x)\,dx\le s  \E^D(f,f)+\beta (s)\Big(\int_D \phi_1(x)|f(x)|\,dx\Big)^2,\quad f \in C_c^\infty(D),\ s>0,
\end{equation*}
where
\begin{equation}\label{fuc}\beta(r)=\frac{4\beta_0(\frac{r}{2})}{\Theta(1/(4\beta_0(\frac{r}{2})))^2},\,\,\Theta(r)=\sup\Big\{s>0: \big|\{x \in D:\ \phi_1(x)\le s\}\big|\le r\Big\}.\end{equation}

If moreover
$$\Psi(r)=\int_r^\infty \frac{\beta^{-1}(s)}{s}\,ds<\infty,\quad r\ge 1,$$ then the associated Dirichlet semigroup $(T_t^D)_{t \ge 0}$ is
intrinsically ultracontractive, and for some constant $c_1>0$,
\begin{equation*}
p^D(t,x,y)\le  c_1\Psi^{-1}(t\wedge 1)e^{-\lambda_1t} \phi_1(x)\phi_1(y),
\end{equation*}
where $-\lambda_1<0$ is the eigenvalue associated with the ground state $\phi_1$.
\end{lemma}

According to Lemma \ref{l4-1}, in order to prove Theorem \ref{t1-2} one only need to derive upper bound of $\beta_0(s)$ in the super Poincar\'{e} inequality \eqref{l4-1-1}, and lower bound of $\Theta(r)$ defined by \eqref{fuc}. First, we have
\begin{lemma}\label{l4-3} Let $X$ be a symmetric L\'{e}vy process given in Section \ref{section11}. Then, the super Poincar\'{e} inequality \eqref{l4-1-1} holds with
$$\beta_0(r)=\Phi_0(r),\quad r>0,$$ where $\Phi_0$ is given in \eqref{symbol-1}. \end{lemma}
\begin{proof}
By our assumption, the transition density $p(t,x,y)$ satisfies that
 $$\sup_{x,y\in \R^d} p(t,x,y)\le c(t),\quad t>0.$$ As mentioned in the proof of Theorem \ref{t1-1},
$$c(t)\le (2\pi)^{-d}\int e^{-t|q(\xi)|}\,d\xi=\Phi_0(t),\quad t>0,$$ see e.g.\ \cite{KSc2}. Then, the desired assertion follows from the estimate above and  \cite[Theorem 3.3(2)]{Wang02} (or \cite[Theorem 3.3.15]{WBook}). \end{proof}

Next, we turn to lower bound estimate for the ground state, which seems to be interesting of itself.

\begin{lemma}\label{l4-2} Let $X$ be a (not necessarily symmetric) L\'evy process such that
{\bf (A2)} is satisfied, and let $D$ be a bounded domain. Then there is a constant $c_1>0$ such that
\begin{equation}\label{l4-2-2}
\phi_1(x)\ge \frac{c_1}{\Phi_1(\frac{1}{\rho_{\partial D}(x)})},\quad x \in D,
\end{equation}where $\Phi_1$ is given in \eqref{symbol-1}.

If moreover {\bf({LDI})} holds for $D$, then for any $\theta>0$, there exists constants $c_2,c_3>0$ such that
$$\Theta(r)\ge \frac{c_2}{\Phi_1\big(e^{c_3r^{-\frac{1}{\theta}}}\big)},\quad r>0.$$
\end{lemma}

\begin{proof} (1)
Since $\phi_1$ is continuous and strictly positive on $D$, it suffices to show
(\ref{l4-2-2}) holds outside some compact subset of $D$.
 According to
\cite[Theorem 5.1]{BSW}, there is a constant $C_1>0$ such that for every $r>0$, $t>0$ and $x \in \R^d$
\begin{equation}\label{l4-2-3}
\Pp^x \big(\tau_{B(x,r)}>t\big)\ge 1 -C_1t\sup_{|\xi|\le \frac{1}{r}}|q(\xi)|=
1 -C_1t\Phi_1\big(\frac 1{r}\big).
\end{equation}

Take $B(x_0, 2r_0)\subseteq D$ with some $x_0\in D$ and $r_0>0$. According to {\bf {(A2)}}, there exist constants $0<r_1<r_2\le \frac{r_0}{16}$ such that for every
ball $B(z,r)\subseteq S(r_1,r_2)$,
$\nu\big(B(z,r)\big)>0$. Then, according to the proof of Lemma \ref{l2-1},
\begin{equation}\label{l4-2-1a}
\zeta(r,r_1,r_2):=1\wedge\inf_{z\in\R^d: B(z,r)\subseteq S(r_1,r_2)}\nu(B(z,r))>0.
\end{equation}
Below, we write $\zeta(r)$ for $\zeta(r,r_1,r_2)$, and let $\tilde r:=\frac{r_1+r_2}{2}$.

Since $D$ is bounded and connected, for every $y \in D$
with $\rho_{\partial D}(y)\le \frac{\tilde r}{16}$, we can find finite points $\{y_i\}_{i=1}^n:=\{y_i(y)\}_{i=1}^n\subseteq D$
with some positive integer $n:=n(y)$, such that the following properties hold true:
\begin{itemize}
\item[(i)] $y_1=y$.

\item[(ii)] $|y_{i+1}-y_i|=\tilde r$ for every $1 \le i \le n-1$.

\item[(iii)] There exists a constant $0<\varepsilon<\frac{1}{16} \wedge\frac{r_2-r_1}{3(r_1+r_2)}$ independent of $y$ such that
$B(y_i, 2\varepsilon \tilde r)\subseteq D$ for every $2\le i\le n-1$, and
$B(y_n, 2\varepsilon \tilde r)\subseteq B(x_0,r_0)$.

\item[(iv)] There exists a positive integer $N$ such that $$\sup\Big\{n(y): y \in D, \rho_{\partial D}(y)\le \frac{\tilde r}{16}\Big\} \le N.$$

\end{itemize}
In the following, for any $y\in D$ with $\rho_{\partial D}(y)\le \frac{\tilde r}{16}$. Let $D_1:=B\big(y,\rho_{\partial D}(y)\big)$, $\tilde D_i:=B\big(y_i, \varepsilon \tilde r\big)$ and
$D_i:=B\big(y_i,2\varepsilon \tilde r\big)$ for every $2\le i\le n$. Define a sequence of
stopping times as follows
\begin{equation*}
\begin{split}
\tilde \tau_{D_1}:=\tau_{D_1},\,\, \tilde \tau_{D_{i+1}}:=\inf\{t>\tilde \tau_{D_i}: X_t \notin D_{i+1}\},\quad  i\ge 1.
\end{split}
\end{equation*}
Set $t_0:=\frac{1}{4C_1\Phi_1(\frac1{\varepsilon \tilde r})}$ and $t_1(y):=\frac{1}{4C_1\Phi_1(\frac{1}{\rho_{\partial D}(y)})}$.

For any $y\in D$ with $\rho_{\partial D}(y)\le \min\big\{ \varepsilon ,\frac{1}{16}\big\}\,\tilde r$, we have
$t_1(y)\le t_0$, and
\begin{align*}
&T_{2t_0}^D(\I_{B(x_0,r_0)})(y)\\
&\ge T^D_{2t_0}(\I_{D_n})(y)\\
&=\Ee^y\Big(\I_{D_n}(X^D_{2t_0})\Big)\\
&\ge \Pp^y\Big(0<\tilde\tau_{D_1}<t_1(y),0<\tilde{\tau}_{D_{i}}-\tilde{\tau}_{D_{i-1}}<\frac{t_0}{n},\\
&\qquad\quad  X^D_{\tilde{\tau}_{D_{i-1}}}\in \tilde D_{i}\
{\rm for\ each}\ 2\le i \le n, \textrm{ and } \forall_{s\in[\tilde{\tau}_{D_n},2t_0] } X^D_{s} \in D_n\Big)\\
& =\Pp^y\Big(0<\tilde\tau_{D_1}<t_1(y),0<\tilde{\tau}_{D_{i}}-\tilde{\tau}_{D_{i-1}}<\frac{t_0}{n},\\
&\qquad\quad  X_{\tilde{\tau}_{D_{i-1}}}\in \tilde D_{i}\
{\rm for\ each}\ 2\le i \le n, \textrm{ and } \forall_{s\in[\tilde{\tau}_{D_n},2t_0] } X_{s} \in D_n\Big)\\
& = \Pp^y\Big(\I_{\{0<{\tau}_{D_{1}}<{t_1(y)}, X_{{\tau}_{D_{1}}}\in
\tilde D_{2}\}}\cdot\Pp^{X_{\tilde{\tau}_{D_{1}}}}\Big(0<\tau_{D_{2}}<\frac{t_0}{n},X_{{\tau}_{D_{2}}}\in
\tilde D_{3};\\
&\qquad\quad\,\cdot \Pp^{X_{\tilde{\tau}_{D_{2}}}}\Big(\cdots
\Pp^{X_{\tilde{\tau}_{D_{n-2}}}}\Big(0<{\tau}_{D_{n-1}}<\frac{t_0}{n},X_{\tau_{D_{n-1}}}\in
\tilde D_n;\\
&\qquad\quad\,\,\cdot
\Pp^{X_{\tilde{\tau}_{D_{n-1}}}}\Big(\forall_{s\in[0,2t_0-\tilde {\tau}_{D_n}]} X_{s}
\in D_n\Big)\Big)\cdots\Big)\Big)\Big),
\end{align*}
where the last equality follows from the strong Markov property.

By (\ref{l4-2-3}),  for every $x \in D$,
\begin{equation}\label{l4-2-4-1}
\begin{split}
&\Pp^x
\big(X_t \in B(x,\varepsilon \tilde r) \textrm{ for all }0 < t \le 2t_0\big)
\ge \Pp^x(\tau_{B(x,\varepsilon \tilde r)}>2t_0)\ge \frac{1}{2},
\end{split}
\end{equation} which gives us that
\begin{equation*}
\begin{split}\Pp^{X_{\tilde{\tau}_{D_{n-1}}}}\Big(\forall_{s\in[0,2t_0-\tilde {\tau}_{D_n}]} X_{s}
\in D_n\Big)\ge \inf_{x\in \tilde{D}_n}\Pp^x\Big(X_t\in B(x,\varepsilon \tilde r) \textrm{ for all }0 < t \le 2t_0\Big)\ge \frac{1}{2}.  \end{split}
\end{equation*}

On the other hand, for any
$2\le i\le n-1$, if $X_{\tilde{\tau}_{D_{i-1}}} \in \tilde D_{i}$, then, by the L\'{e}vy system of the process $X$, see Lemma \ref{l-p-1},
\begin{equation}\label{l4-2-3a}
\begin{split}
&\Pp^{X_{\tilde{\tau}_{D_{i-1}}}}\Big(0<{\tau}_{D_{i}}<\frac{t_0}{n},X_{{\tau}_{D_{i}}}\in
\tilde D_{i+1}\Big)\\
&\ge \inf_{y \in \tilde D_i}\int_0^{\frac{t_0}{n}}\int_{D_i}p^{D_i}(s,y,z)\Big(
\int_{\tilde D_{i+1}-z}\nu(dw)\Big)\,dz\,ds\\
&\ge \frac{t_0}{n}\Big(\inf_{y \in \tilde D_i}\Pp^y(\tau_{D_i}>\frac{t_0}{n})
\Big)\Big(\inf_{z \in D_i}\nu\big(B(y_{i+1}-z,\varepsilon \tilde r)\big)\Big).
\end{split}
\end{equation}
For every $w \in B(y_{i+1}-z,\varepsilon \tilde r)$ and $z \in D_i$, it holds that
$$|w-(y_{i+1}-y_i)|\le |w-(y_{i+1}-z)|+|z-y_i|\le 3\varepsilon \tilde r,$$ and so
$$r_1\le |y_{i+1}-y_i|-3\varepsilon \tilde r\le|w| \le |y_{i+1}-y_i|+3\varepsilon \tilde r\le r_2.$$ This implies that $$B(y_{i+1}-z,\varepsilon \tilde r)\subseteq S(r_1,r_2).$$ Using \eqref{l4-2-1a} and \eqref{l4-2-4-1}, we find that the right hand side of \eqref{l4-2-3a} is bigger than
\begin{equation*}
\begin{split} \frac{t_0 \zeta(\varepsilon \tilde r)}{N}\Big(\inf_{y \in \tilde D_i}\Pp^y(\tau_{B(y,\varepsilon \tilde r)}>t_0)
\Big)\ge \frac{3t_0 \zeta(\varepsilon \tilde r)}{4N}.\end{split}\end{equation*}

Similarly, we can obtain that
for every $y \in D$ with $\rho_{\partial D}(y)\le  \min\big\{ \varepsilon ,\frac{1}{16}\big\}\,\tilde r$,
\begin{align*}
\Pp^y\Big(0<\tau_{D_1}<t_1(y),X_{\tau_{D_1}}\in \tilde D_2\Big)
&\ge t_1(y) \zeta(\varepsilon \tilde r)
\Pp^y\Big(\tau_{B(y,\rho_{\partial D}(y))}>t_1(y)\Big)\\
&\ge \frac{3t_1(y)\zeta(\varepsilon \tilde r)}{4}\\
&=  \frac{3\zeta(\varepsilon \tilde r)}{16C_1\Phi_1(\frac{1}{\rho_{\partial D}(y)})}.
\end{align*}
Combining all the estimates above yields that for
every $y \in D$ with $\rho_{\partial D}(y)\le  \min\big\{ \varepsilon ,\frac{1}{16}\big\}\,\tilde r$,
\begin{equation*}
\begin{split}
& T_{2t_0}^D(\I_{B(x_0,r_0)})(y)\ge
\frac{C_2\big(\frac{\zeta(\varepsilon \tilde r)}{2N}\big)^N\zeta(\varepsilon \tilde r)}{\Phi_1(\frac{1}{\rho_{\partial D}(y)})}
\ge \frac{C_3}{\Phi_1(\frac{1}{\rho_{\partial D}(y)})}.
\end{split}
\end{equation*}

Therefore, for every $y \in D$ with $\rho_{\partial D}(y)\le  \min\big\{ \varepsilon ,\frac{1}{16}\big\}\,\tilde r$
\begin{equation*}
\begin{split}
\phi_1(y)&=e^{2\lambda_1 t_0}T_{2t_0}^D(\phi_1)(y)\ge {e^{2\lambda_1 t_0}}\Big({\inf_{z \in B(x_0,r_0)}\phi_1(z)}\Big) T_{2t_0}^D(\I_{B(x_0,r_0)})(y)
\ge
\frac{C_4}{\Phi_1(\frac{1}{\rho_{\partial D}(y)})},
\end{split}
\end{equation*}
which proves (\ref{l4-2-2}).

(2) Suppose that $D$ satisfies {\bf{(LDI)}}.
Then, \eqref{e1-4} and the Chebyshev inequality yield that for each $\theta,s>0$,
\begin{equation*}
\begin{split}
\big|\big\{x \in D: \rho_{\partial D}(x)\le s\}\big|& =\Big|\{x \in D: \log\Big(\frac{1}{\rho_{\partial D}(x)}
\Big)\ge c_1\log \frac{1}{s}\big\}\Big|\\
&\le  \frac{\int_D \Big|\log\Big(\frac{1}{\rho_{\partial D}(x)}
\Big)\Big|^{\theta}dx}{|\log {s}|^\theta}\le \frac{C_6}{|\log s|^\theta}.
\end{split}
\end{equation*}
Therefore, by \eqref{l4-2-2} and the increasing of $\Phi_1(r)$, we get that for any $r, \theta>0$,
\begin{equation*}
\begin{split}
\Theta(r)&=\sup\Big\{s>0: \big|\{x \in D: \phi_1(x)< s\}\big|\le r\Big\}\\
&\ge \sup\Big\{s>0: \Big|\big\{x \in D: \Phi_1\Big(\frac 1{\rho_{\partial D}(x)}\Big)\ge \frac{{c_1}}{s}\big\}\Big|\le r\Big\}\\
&\ge \sup\Big\{s>0: \Big|\big\{x \in D:  \rho_{\partial D}(x)\le \Big[\Phi_1^{-1}(\frac{{c_1}}{s})\Big]^{-1}\big\}\Big|\le r\Big\}\\
&\ge \sup\Big\{s>0: \frac{C_6}{\log^\theta \Phi_1^{-1}(\frac{{c_1}}{s})}\le r \Big\}\\
&\ge \frac{C_7}{\Phi_1\Big(e^{C_8r^{-\frac{1}{\theta}}}\Big)},
\end{split}
\end{equation*}
where $C_7, C_8$ are positive constants depending on $\theta$. The proof is complete.
\end{proof}
\begin{remark} According to \cite[Theorem 1.1]{CSK}, \eqref{l4-2-2} is not optimal for symmetric $\alpha$-stable process on bounded $C^{1,1}$-domain.  However, as stated in Lemma \ref{l4-2}, \eqref{l4-2-2} holds for general (not necessarily symmetric) L\'{e}vy process and bounded domain with any regularity condition.
\end{remark}

\begin{proof}[Proof of Theorem $\ref{t1-2}$]Having Lemmas \ref{l4-1}, \ref{l4-2} and \ref{l4-3} at hand, one can immediately obtain Theorem \ref{t1-2}.
\end{proof}

At the end of this section, we present the

\begin{proof}[Proof of Example $\ref{exm2}$] By the definition of L\'{e}vy measure $\nu$, it is clear that assumption {\bf (A2)} holds.

By some element calculations, one can get that there is a constant $c_2\ge 1$ such that for $r>0$ small enough
$$c_2^{-1}r^{-\alpha}\le r^{-2}\int_{\{|z|\le r\}} |z|^2\,\nu(dz)\le \int\Big(1\wedge\frac{|z|}{r}\Big)^2\,\nu(dz)\le c_2r^{-\alpha}.$$
According to \cite[Proposition 1 and Lemma 5]{KSz} and the inequality above, there exists a constant $c_3\ge 1$ such that
for $r>0$ large enough,
$$c_3^{-1} r^{\alpha}\le \Phi_1(r)\le c_3 r^{\alpha},$$
for $r>0$ small enough
$$c_3^{-1} r^{-d/\alpha}\le \Phi_0(r)\le c_3 r^{-d/\alpha}.$$
Furthermore, on the one hand, by \cite[Theorem 1]{KSc2}, we know that the process $X$ has transition density $p(t,x,y)=p(t,0,y-x)$ such that for each $t>0$, $(x,y)\mapsto p(t,x,y): \R^d\times\R^d\to[0,\infty)$ is continuous. On the other hand, \cite[Lemmas 5 and 7]{KSz} and the Chapman-Kolmogorov equation for transition density $p(t,x,y)$ yield that for any $t>0$ and $x,y\in\R^d$, $p(t,x,y)>0$. Combining with all the conclusions above, we find that $p(t,x,y)$ satisfies all the assumptions in Subsection \ref{section11}. Besides, it is easy to see that all the assumptions for $p^D(t,x,y)$ also hold true, thanks to \cite[Proposition 2.2(i)]{G}.

The above estimates for $\Phi_0$ and $\Phi_1$ imply that for any $\theta>0$, there are constants $c_4,c_5>0$ such that for $r>0$ small enough,
$$\beta(r)\le c_4\exp\big(c_5(1+r^{-\frac{d}{\alpha \theta}})\big).$$
Whence, the desired assertions follow from Theorem \ref{t1-2}.
\end{proof}

\noindent{\bf Acknowledgements.} The authors are indebted to the referee for his/her careful corrections. The authors also would like to thank Professors Panki Kim,  Renming Song and Feng-Yu Wang for their helpful comments on previous versions of the paper. Financial support through \lq\lq Yang Fan Project\rq\rq \, of Science and Technology Commission of Shanghai Municipality (No.\ 15YF1405900) (for Xin Chen),  National
Natural Science Foundation of China (No.\ 11201073), the JSPS postdoctoral fellowship
(26$\cdot$04021), National Science Foundation of Fujian Province (No.\ 2015J01003), and the Program for Nonlinear Analysis and Its Applications (No.\ IRTL1206)
(for Jian Wang) is gratefully acknowledged.

\end{document}